\documentclass[sn-mathphys]{sn-jnl}
\jyear{2021}%

\usepackage{amsmath}
\theoremstyle{thmstyleone}%
\newtheorem{theorem}{Theorem}
\newtheorem{proposition}[theorem]{Proposition}%
\newtheorem{corollary}[theorem]{Corollary}

\newtheorem{schneider}[theorem]{Schneider's Theorem}

\theoremstyle{thmstyletwo}%
\newtheorem{example}[theorem]{Example}%
\newtheorem{remark}[theorem]{Remark}%

\theoremstyle{thmstylethree}%
\newtheorem{definition}[theorem]{Definition}%

\DeclareMathOperator{\arcsl}{arcsl}
\DeclareMathOperator{\arcs}{arcs}

\begin{document}

\title[Hyperelliptic values of the Gamma function]
{Hyperelliptic values of the Gamma function}


\author{\fnm{Jan} \sur{Lügering}}


\abstract{We show how to calculate particular values of the Gamma function for specific rational arguments in the interval (0,1) without using the Elliptic K-function. Instead we use transcendental constants or periods defined by hyperelliptic integrals.}





\maketitle

\section{Introduction}\label{sec1}
The Gamma function can be defined for $Re(x) > 0$ as
\begin{equation}
\Gamma(x) := \int \limits_{0}^{\infty} t^{x-1} \cdot e^{-t} \ dt
\end{equation}
Recall the Legendre duplication formula for $x > 0$:
\begin{equation}
\Gamma(x) \cdot \Gamma(x + \frac{1}{2}) = \frac{\sqrt{\pi}}{2^{2x-1}} \cdot \Gamma(2x)
\end{equation}
and Gauss multiplication formula for $x \in \mathbb{C} \setminus -\mathbb{N}_{0}$:
\begin{equation}
\forall \ n \in \mathbb{N}: \ \Gamma(x) \cdot \Gamma(x + \frac{1}{n}) \cdot ... \cdot \Gamma(x + \frac{n-1}{n}) = \frac{(2 \pi)^{\frac{n-1}{2}}}{n^{nx - \frac{1}{2}}} \cdot \Gamma(nx)
\end{equation}
The Euler Beta function can be defined for $Re(w)>0$ and $Re(z)>0$ as:
\begin{equation}
B(w,z) := \int \limits_{0}^{1} t^{w-1} \cdot (1-t)^{z-1} \ dt
\end{equation}
The central identity in the theory of the Beta function is \cite{remmert}:
\begin{equation}
\forall \ Re(w),Re(z) > 0: \ B(w,z) = \frac{\Gamma(w) \cdot \Gamma(z)}{\Gamma(z+w)}
\end{equation}
We also have the Euler relation:
\begin{equation}
\forall \ x \in \mathbb{C} \setminus \mathbb{Z}: \ \Gamma(x) \cdot \Gamma(1-x) = \frac{\pi}{\sin(\pi x)} 
\end{equation}
The lemniscatic arcsine function is defined as \cite{siegel}:
\begin{equation}
\arcsl(x) := \int \limits_{0}^{x} \frac{1}{\sqrt{1-r^4}} \ dr
\end{equation}
with the lemniscate constant
\begin{equation}
\varpi := 2 \cdot \int \limits_{0}^{1} \frac{1}{\sqrt{1-r^4}} \ dr
\end{equation}
in analogy to
\begin{equation}
\arcsin(x) := \int \limits_{0}^{x} \frac{1}{\sqrt{1-r^2}} \ dr
\end{equation}
and
\begin{equation}
\pi := 2 \cdot \int \limits_{0}^{1} \frac{1}{\sqrt{1-r^2}} \ dr
\end{equation} \\
We have the fascinating integral identity \cite{euler}:
\begin{equation}
A \cdot B := (\int \limits_{0}^{1} \frac{1}{\sqrt{1-r^4}} \ dr) \cdot (\int \limits_{0}^{1} \frac{r^2}{\sqrt{1-r^4}} \ dr) = \frac{\pi}{4}
\end{equation} 
Our goal is to generalize these functions and identities and apply this knowledge to calculate fractional values of the Gamma function.\\ \\ \\
\section{Particular values of the Gamma function}\label{sec2}
\begin{definition}\label{def1}
We define for $s \geq 1$:
\begin{equation}
\pi_s := 2 \cdot \int \limits_{0}^{1} \frac{1}{\sqrt{1-x^s}} \ dx
\end{equation}
\end{definition}
\begin{definition}\label{df2}
The generalised arcsine is defined for $n \in \mathbb{N}$ as:
\begin{equation}
\arcs_{n}(x) := \int \limits_{0}^{x} \frac{1}{\sqrt{1-r^n}} \ dr, \ x \in V_n := \left\{\begin{array}{lr}
        [-1,1] & \text{for n even}\\
        (-\infty, 1], & \text{for n odd}\\
        \end{array}\right\}
\end{equation}
\end{definition}
\begin{proposition}\label{pr3}
\begin{equation}
\forall \ \lvert x \rvert < 1: \arcs_n(x) = \sum_{k=0}^{\infty} \binom{2k}{k} \frac{x^{nk+1}}{4^{k} \cdot (nk+1)}
\end{equation}
and in particular we have
\begin{equation}
\frac{\pi_n}{2} = \sum_{k=0}^{\infty} \binom{2k}{k} \frac{1}{4^{k} \cdot (nk+1)}
\end{equation}
\end{proposition}
\begin{proof}
Let $ \lvert x \rvert \leq r < 1$ be arbitrary.
\begin{equation*}
\arcs_n(x) = \int \limits_{0}^{x} \sum_{k=0}^{\infty} \binom{-1/2}{k} \cdot (-1)^{k} \cdot y^{nk} \ dy = \int \limits_{0}^{x} \sum_{k=0}^{\infty} \binom{2k}{k} \cdot \frac{y^{nk}}{4^{k}} \ dy
\end{equation*}
Since $\sum_{k=0}^{\infty} \binom{2k}{k} \cdot \frac{y^{nk}}{4^{k}}$ converges uniformally for $\lvert y \rvert \leq r$ we obtain:
\begin{equation*}
= \sum_{k=0}^{\infty} \binom{2k}{k} \frac{x^{nk+1}}{4^{k} \cdot (nk+1)} \ , \forall \ \lvert x \rvert \leq r
\end{equation*}
For every fixed $ \lvert x \rvert < 1$  and in particular for $x=1$, the series converges by comparison with the arcsine series. Abel's limit theorem yields the result.
\end{proof}
Our next goal is to find a halving formula by proving an auxiliary formula.
\begin{proposition}\label{pr4}
\begin{equation}
\forall n \in \mathbb{N}: \ \frac{\Gamma(\frac{1}{n})^{2} \cdot 2^{\frac{2}{n} - 1}}{\Gamma(\frac{2}{n}) \cdot n} = \frac{\pi_n}{2}
\end{equation}
\end{proposition}
\begin{proof}
Let $m,n \in \mathbb{N}$. \\ \\
On the one hand we have:
\begin{equation*}
B(\frac{1}{2}, \frac{m}{n}) = \frac{\sqrt{\pi} \cdot \Gamma(\frac{m}{n})}{\Gamma(\frac{m}{n} + \frac{1}{2})}
\end{equation*}
On the other hand:
\begin{equation*}
B(\frac{1}{2}, \frac{m}{n}) = B(\frac{m}{n}, \frac{1}{2}) = \int \limits_{0}^{1} t^{\frac{m}{n} - 1} \cdot (1-t)^{- \frac{1}{2}} \ dt
\end{equation*}
Substituting $t := u^n$ yields:
\begin{equation*}
= n \cdot \int \limits_{0}^{1} \frac{u^{m-1}}{\sqrt{1 - u^{n}}} \ du
\end{equation*}
Therefore 
\begin{equation*}
\frac{\sqrt{\pi} \cdot \Gamma(\frac{m}{n})}{\Gamma(\frac{m}{n} + \frac{1}{2})} = n \cdot \int \limits_{0}^{1} \frac{u^{m-1}}{\sqrt{1 - u^{n}}} \ du
\end{equation*}
For $m=1$ we obtain:
\begin{equation*}
\frac{\sqrt{\pi} \cdot \Gamma(\frac{1}{n})}{\Gamma(\frac{1}{n} + \frac{1}{2}) \cdot n} = \int \limits_{0}^{1} \frac{1}{\sqrt{1 - u^{n}}} \ du = \frac{\pi_n}{2}
\end{equation*}
\begin{equation*}
\implies \frac{\sqrt{\pi}}{\Gamma(\frac{1}{n} + \frac{1}{2})} = \frac{n \cdot \pi_n}{\Gamma(\frac{1}{n}) \cdot 2}
\end{equation*} \\
Rearranging the Legendre duplication formula for $x > 0$ yields:
\begin{equation*}
\frac{\sqrt{\pi}}{\Gamma(x + \frac{1}{2})} = \frac{\Gamma(x) \cdot 2^{2x - 1}}{\Gamma(2x)}
\end{equation*}
Setting $x = \frac{1}{n}$:
\begin{equation*}
\frac{\sqrt{\pi}}{\Gamma(\frac{1}{n} + \frac{1}{2})} = \frac{\Gamma(\frac{1}{n}) \cdot 2^{\frac{2}{n} - 1}}{\Gamma(\frac{2}{n})}
\end{equation*} \\
Comparing both equations implies:
\begin{equation*}
\frac{n \cdot \pi_n}{\Gamma(\frac{1}{n}) \cdot 2} = \frac{\Gamma(\frac{1}{n}) \cdot 2^{\frac{2}{n} - 1}}{\Gamma(\frac{2}{n})}
\end{equation*}
\end{proof}
\begin{theorem}\label{th5}
For $s \geq 1$ we have:
\begin{equation}
\Gamma(\frac{1}{2s}) = \sqrt{\frac{\pi_{2s} \cdot \Gamma(\frac{1}{s}) \cdot 2s}{2^{\frac{1}{s}}}}
\end{equation}
\end{theorem}
\begin{proof}
Let $l \in \mathbb{N}$. \\
In Proposition $\ref{pr4}$ set $n = 2l$ and rearrange. Observe that its proof also holds for $n \in \mathbb{R}_{\geq 1}$.
\end{proof}
\begin{corollary}\label{co6}
\begin{equation}
\forall n \in \mathbb{N}: \ \Gamma(\frac{1}{2^{n}}) = \prod_{k=1}^{n} (2^{n-1} \cdot \pi_{2^{k}})^{\frac{1}{2^{(n-k+1)}}}
\end{equation}
\end{corollary}
\begin{proof}
Let $n \in \mathbb{N}$. \\ \\
Initial case) For $n=1$: $\Gamma(\frac{1}{2}) = \sqrt{\pi_2} = \sqrt{\pi}$ \\ \\
Induction hypothesis) Suppose that our claim holds for arbitrary, but once chosen, fixed $n \in \mathbb{N}$. \\ \\
Induction step) For $n \mapsto n+1$: 
\begin{equation*}
\Gamma(\frac{1}{2^{n+1}}) = \sqrt{\pi_{2^{n+1}} \cdot \Gamma(\frac{1}{2^n}) \cdot 2^n \cdot 2^{\frac{2^n - 1}{2^n}}}
\end{equation*}
\begin{equation*}
= \sqrt{\pi_{2^{n+1}} \cdot 2^{n+1 - \frac{1}{2^n}}} \cdot \sqrt{\prod_{k=1}^{n} (2^{n-1} \cdot \pi_{2^{k}})^{\frac{1}{2^{(n-k+1)}}}}
\end{equation*} \\
The 2 at the left now transfers the powers $\sum_{k=1}^{n} (\frac{1}{2})^k = 1 - \frac{1}{2^n}$ into the product. \\
\begin{equation*}
= \sqrt{\pi_{2^{n+1}} \cdot 2^{n}} \cdot \sqrt{\prod_{k=1}^{n} (2^{n} \cdot \pi_{2^{k}})^{\frac{1}{2^{(n-k+1)}}}} \ = \ \prod_{k=1}^{n+1} (2^{n} \cdot \pi_{2^{k}})^{\frac{1}{2^{(n-k+2)}}}
\end{equation*} \\
By the principle of induction, our claim holds for all $n \in \mathbb{N}$.
\end{proof}
\begin{example}\label{ex7}
\begin{equation}
\Gamma(\frac{1}{2}) = \sqrt{\pi_2} = \sqrt{\pi}
\end{equation}
\begin{equation}
\Gamma(\frac{1}{4}) = \sqrt{2 \pi_4 \sqrt{2\pi_2}} = \sqrt{2 \varpi \sqrt{2\pi}}
\end{equation}
The Euler relation yields:
\begin{equation}
\Gamma(\frac{3}{4}) = \frac{\pi}{\sqrt{\varpi \sqrt{2 \pi}}}
\end{equation}
The pattern continues:
\begin{equation}
\Gamma(\frac{1}{8}) = \sqrt{4 \pi_8 \sqrt{4 \pi_4 \sqrt{4\pi_2}}}
\end{equation} \\
In general we obtain for the "quarter-length" of the "unit curve" $q_n:= \frac{\pi_n}{2}$: \\
\begin{equation}
\Gamma(\frac{1}{2^{n}}) = \sqrt{2^n \cdot q_{2^n} \cdot \sqrt{2^n \cdot q_{2^{(n-1)} \cdot \sqrt{... \cdot \sqrt{2^n \cdot q_{2}}}}}}
\end{equation}
\end{example}
We state the following definition and theorem without proof as it follows the same reasoning as before. We will never use it again in this paper.
\begin{definition}\label{def8}
For $l,n,k \in \mathbb{N}$ and $l \geq 2$ and $k < l$ we define
\begin{equation}
\pi_{l,n,k} := 2 \cdot \int \limits_{0}^{1} \frac{1}{\sqrt[l]{(1-t^n)^{k}}} \ dt
\end{equation}
\end{definition}
\begin{theorem}\label{th10}
\begin{equation}
\Gamma(\frac{1}{l \cdot n}) = \sqrt{\frac{\Gamma(\frac{1}{n}) \cdot l \cdot n \cdot (2 \pi)^{\frac{l-1}{2}} \cdot l^{\frac{1}{2} - \frac{1}{n}} \cdot \pi_{l,l \cdot n,l-k}}{\Gamma(\frac{k}{l}) \cdot 2 \cdot \prod \limits_{j=1,j \neq k}^{l-1} \Gamma(\frac{1}{l \cdot n} + \frac{j}{l})}}
\end{equation}
For $n,k,l \in \mathbb{N}$ and $l \geq 2$, where $\ 1 \leq k \leq l-1$ is fixed.
\end{theorem}

\begin{remark}\label{re13}
\begin{equation}
\Gamma(\frac{1}{3}) = \sqrt[3]{\sqrt[2]{3} \cdot \pi_{3} \cdot \pi_{2} \cdot \sqrt[3]{2}}
\end{equation}
\end{remark}
\begin{proof}
In the proof of Proposition $\ref{pr4}$ we deduced: 
\begin{equation}\label{eqkk}
\frac{\sqrt{\pi} \cdot \Gamma(\frac{1}{n})}{\Gamma(\frac{1}{n} + \frac{1}{2}) \cdot n} = \int \limits_{0}^{1} \frac{1}{\sqrt{1 - u^{n}}} \ du
\end{equation}
Setting $n=3$: 
\begin{equation*}
\frac{\sqrt{\pi} \cdot \Gamma(\frac{1}{3})}{\Gamma(\frac{5}{6}) \cdot 3} = \int \limits_{0}^{1} \frac{1}{\sqrt{1 - u^{3}}} \ du
\end{equation*}
By Legendres duplication formula we have:
\begin{equation*}
\Gamma(\frac{5}{6}) = \Gamma(\frac{1}{3} + \frac{1}{2}) = \frac{\frac{\sqrt{\pi}}{2^{\frac{2}{3} - 1}} \cdot \Gamma(\frac{2}{3})}{\Gamma(\frac{1}{3})}
\end{equation*}
Applying the Euler relation, we get in total: 
\begin{equation}\label{eqk}
\frac{\Gamma(\frac{1}{3})^3}{\sqrt{3} \cdot \pi \cdot \sqrt[3]{2^4}} = \int \limits_{0}^{1} \frac{1}{\sqrt{1 - u^{3}}} \ du = \frac{\pi_3}{2}
\end{equation}
\end{proof}
The integral identities $\ref{eqk}$ and $\ref{eqkk}$ were remarked by Konrad Königsberger and can be found in his book \cite{koenigsberger}. \\
In the book \cite{freitag} by Freitag and Busam, we find the identity 
\begin{equation}
\Gamma(\frac{1}{6}) = 2^{- \frac{1}{3}} \cdot (\frac{3}{\pi})^{\frac{1}{2}} \cdot \Gamma(\frac{1}{3})^2
\end{equation}
Applying the halving formula to $\ref{re13}$ we obtain
\begin{equation}
\Gamma(\frac{1}{6}) = \sqrt[2]{\pi_6 \cdot 3 \cdot \sqrt[3]{4 \cdot \sqrt[3]{2} \cdot \sqrt[2]{3} \cdot \pi_2 \cdot \pi_3}}
\end{equation}
plugging in the values into the identity, we obtain
\begin{remark}\label{re14}
\begin{equation}
\pi_6 = \sqrt{3} \cdot \frac{\pi_3}{2}
\end{equation}
\end{remark}
This is plausible, because by definition
\begin{remark}\label{re15}
\begin{equation}
\lim_{n \to \infty}  q_n = \lim_{n \to \infty} \frac{\pi_n}{2} = 1
\end{equation}
\end{remark}

\section{Table of Values}
Note that all the appearing constants $\pi_n$ are transcendental and that the fractional constant $\pi_{5/2}$ is also transcendental.
\begin{equation}
\Gamma(\frac{1}{2}) = \sqrt{\pi_2} = \sqrt{\pi}
\end{equation}
\begin{equation}
\Gamma(\frac{1}{3}) = \sqrt[3]{\sqrt[2]{3} \cdot \pi_{3} \cdot \pi_{2} \cdot \sqrt[3]{2}}
\end{equation}
\begin{equation}
\Gamma(\frac{1}{4}) = \sqrt{2 \pi_4 \sqrt{2\pi_2}} = \sqrt{2 \varpi \sqrt{2\pi}}
\end{equation}
\begin{equation}
\Gamma(\frac{1}{5}) = \sqrt[5]{\frac{5^3 \cdot \pi_{5}^{2} \cdot \pi_{5/2} \cdot \pi_2}{2^{\frac{13}{5}} \cdot \sqrt{\frac{5 - \sqrt{5}}{8}}}}
\end{equation}
\begin{equation}
\Gamma(\frac{1}{6}) = \sqrt[2]{\frac{\sqrt{3}}{2} \cdot {\pi_3} \cdot 3 \cdot \sqrt[3]{4 \cdot \sqrt[3]{2} \cdot \sqrt[2]{3} \cdot \pi_2 \cdot \pi_3}}
\end{equation}
\begin{equation}
\Gamma(\frac{1}{8}) = \sqrt{4 \pi_8 \sqrt{4 \pi_4 \sqrt{4\pi_2}}}
\end{equation}

\section{Transcendence}
Now what is the nature of the constants $\pi_s$? Are they irrational, or even transcendental?
Thanks to a theorem by Theodor Schneider in \cite{schneider}, we immediately conclude the answer for the rational case and in particular for $n \geq 2$.
\begin{schneider}\label{th17}
For all $\ a, b \in \mathbb{Q} \setminus \mathbb{Z}$ it holds, that $B(a,b)$ is transcendental.
\end{schneider}
\begin{corollary}\label{co18}
For all $n \geq 2$ it holds, that $\pi_n$ is transcendental.
\end{corollary}
\begin{proof}
Let $n \geq 2$ be arbitrary. \\ \\
In the proof of proposition $\ref{pr4}$, we deduced: 
\begin{equation*}
B(\frac{1}{2}, \frac{m}{n}) = n \cdot \int \limits_{0}^{1} \frac{u^{m-1}}{\sqrt{1 - u^{n}}} \ du
\end{equation*}
\begin{equation*}
\implies \frac{2}{n} \cdot B(\frac{1}{2}, \frac{1}{n}) = \pi_n
\end{equation*}
Suppose, $\pi_n$ would be algebraic.
\begin{equation*}
\implies \pi_n \cdot \frac{n}{2} = B(\frac{1}{2}, \frac{1}{n})
\end{equation*}
algebraic, which contradicts Schneider's theorem.
\end{proof}

\section{Generalising Euler's Arcsine Proof}
\begin{theorem}
For $n \in \mathbb{N}$ we have: \\
\begin{equation}
\frac{\pi_{n}^{2}}{8 \cdot \int \limits_{0}^{1} \frac{t}{\sqrt{1 - t^n}} \ dt} = 1 + \sum \limits_{l=0}^{\infty} \frac{(2l+1)!!}{(2l+2)!!} \cdot \frac{1}{n \cdot (l+1) + 1} \cdot \prod \limits_{k=0}^{l} \frac{2 \cdot (nk+2)}{n + 2 \cdot (nk+2)}
\end{equation}
\end{theorem}

\begin{proof}

We have the Taylor series expansion

\begin{equation}
arcs_n(t) = \sum_{k=0}^{\infty} \binom{2k}{k} \frac{t^{nk+1}}{4^{k} \cdot (nk+1)}
\end{equation}

In other words:

\begin{equation}
arcs_n(t) = t + \sum \limits_{k=1}^{\infty} \frac{(2k-1) \cdot ... \cdot 5 \cdot 3 \cdot 1}{2k \cdot ... \cdot 4 \cdot 2} \cdot (\frac{t^{nk+1}}{(nk+1)})
\end{equation}

Observing that:

\begin{equation}
\frac{1}{2} \cdot (arcs_n(x))^2 = \int \limits_{0}^{x} \frac{arcs_n(t)}{\sqrt{1 - t^n}} \ dt
\end{equation}

We obtain: 
\begin{equation}
\frac{1}{2} \cdot (arcs_n(x))^2 =\int \limits_{0}^{x} \frac{t}{\sqrt{1 - t^n}} + \sum \limits_{k=1}^{\infty} \frac{(2k-1) \cdot ... \cdot 5 \cdot 3 \cdot 1}{2k \cdot ... \cdot 4 \cdot 2} \cdot (\frac{1}{nk+1}) \cdot \int \limits_{0}^{x} \frac{t^{nk+1}}{\sqrt{1 - t^n}} \ dt
\end{equation}

Next, define
\begin{equation}
I_l(x) = \int \limits_{0}^{x} \frac{t^{l + n}}{\sqrt{1-t^n}} \ dt = \int \limits_{0}^{x} t^{l + 1} \cdot \frac{t^{n-1}}{\sqrt{1 - t^n}} \ dt
\end{equation}

Integrating by parts and multipliying the integrand with $\frac{\sqrt{1-t^n}}{\sqrt{1-t^n}}$ : 

\begin{equation}
I_l(x) = - \frac{2}{n} \cdot x^{l+1} \cdot \sqrt{1 - x^n} + \frac{2(l+1)}{n} \cdot \int \limits_{0}^{x} \frac{t^{l} \cdot (1 - t^n)}{\sqrt{1-t^n}} \ dt
\end{equation}

Setting $x = 1$ we obtain:
\begin{equation}
\int \limits_{0}^{1} \frac{t^{l+n}}{\sqrt{1-t^n}} \ dt = \frac{2(l+1)}{n+2(l+1)} \cdot \int \limits_{0}^{1} \frac{t^{l}}{\sqrt{1-t^n}} \ dt
\end{equation}
Now, mapping $l \mapsto (nl+1)$ the result follows.

\end{proof}

This proof suggests that a number $g \leq n$ exists, such that
the system of integrals
\begin{equation}
0 \leq i < g: \ \int \limits_{0}^{1} \frac{x^i}{\sqrt{1 - x^n}} \ dx
\end{equation}
is maximal and $\mathbb{Q}$-linearly independent.

\begin{corollary}
For $n = 2$ we conclude: \\
\begin{equation}
\frac{\pi^2}{8} = \sum \limits_{l=0}^{\infty} \frac{1}{(2l+1)^2}
\end{equation}
And therefore: \\
\begin{equation}
\zeta(2) = \frac{\pi^2}{6}
\end{equation}
\end{corollary}

To compute the integral $\int \limits_{0}^{1} \frac{t}{\sqrt{1 - t^n}} \ dt$ for $n=3$ the following relation by Leo Königsberger might be useful: \cite{lk}
\begin{theorem}
For $p \in \mathbb{N}$:
\begin{equation}
\prod \limits_{s = 0}^{2p-1} (\zeta_{2p+1}^{s+1} - 1) \cdot \int \limits_{0}^{1} \frac{x^s}{\sqrt{x^{2p+1} - 1}} = \frac{(2 \pi)^p \cdot \sqrt{2p+1}}{(2p+1)^p \cdot (2p-1)!!}
\end{equation}
\end{theorem}

\section{Overview: Differential Forms on Algebraic Curves}
We study the family of hyperelliptic curves
\begin{equation}
C_n : \ y^2 = 1 - x^n
\end{equation} \\
The algebraic genus $g$ is defined as the dimension of the space of global holomorphic differentials on $C_n$.
\begin{equation}
g = \lfloor\frac{n-1}{2}\rfloor
\end{equation} \\
A standard basis for the space of holomorphic differentials is given by
\begin{equation}
\omega_k = \frac{x^k dx}{y}, \ \ 0 \leq k < g 
\end{equation}
For $1 \leq k \leq g$ we use the constants
\begin{equation}
\pi_{n / k} = 2k \cdot \int_{0}^{1}\frac{x^{k-1}}{\sqrt{1 - x^n}}\,dx
\end{equation} 
to describe particular values of the Gamma function for rational arguments. \\ \\
These constants correspond to the real A-periods for $0 \leq k < g$
\begin{equation}
\int \limits_{\Gamma_{A}} \omega_{k} = \int \limits_{\gamma +} \omega_{k} - \int \limits_{\gamma -} \omega_{k} = 2 \cdot \int \limits_{0}^{1} \frac{x^k}{\sqrt{1 - x^n}} \ dx
\end{equation}
used to construct the period matrix of the associated compact Riemann surface.

\section{Miscellanea}\label{sec5}
It is well known, that the singular values $K(\lambda^*(n))$ of the complete elliptic integral, where $\lambda^*(n) := \sqrt{\lambda(i \cdot \sqrt{n})}$ and $\lambda$ the modular lambda function, are expressible via the Gamma function \cite{wolfram}.
We obtain the following values:
\begin{equation}
K(\lambda^*(1)) = K(\frac{1}{\sqrt{2}}) = \frac{\varpi}{\sqrt{2}}
\end{equation}
\begin{equation}
K(\lambda^*(3)) = K(\frac{\sqrt{6} - \sqrt{2}}{4}) = \frac{\sqrt{3 \sqrt{3}} \cdot \pi_3}{4}
\end{equation}
\begin{equation}
K(\lambda^*(4)) = K((\sqrt{2} - 1)^2) = \frac{\varpi}{\sqrt{8}} + \frac{\varpi}{4}
\end{equation}
\begin{equation}
K(\lambda^*(9)) = K(\frac{(\sqrt{3}-1) \cdot (\sqrt{2} - \sqrt[4]{3})}{2}) = \sqrt{\frac{\sqrt{3}}{9} + \frac{1}{6}} \cdot \varpi
\end{equation}
\begin{equation}
K(\lambda^*(16)) = K((\sqrt{2} + 1)^2 \cdot (\sqrt[4]{2}-1)^4) = \frac{(\sqrt[4]{2} + 1)^2 \cdot \varpi}{2^3}
\end{equation}
\begin{equation}
K(\lambda^*(25)) = K(\frac{(\sqrt{10} - 2 \cdot \sqrt{2}) \cdot (3 - 2 \cdot \sqrt[4]{5})}{2}) = \frac{\sqrt{2}}{5} \cdot \varpi + \frac{1}{\sqrt{10}} \cdot \varpi
\end{equation} 
\\
Certain limits of the arithmetic-geometric-mean are also expressible via the Gamma function \cite{wikipedia}. We obtain the following values:
\begin{equation}
AGM(1,\sqrt{2}) = \frac{\pi}{\varpi}
\end{equation}
\begin{equation}
AGM(2, \sqrt{2 + \sqrt{3}}) = \frac{4 \cdot \pi_2}{\sqrt[4]{27} \cdot \pi_3}
\end{equation}
\begin{equation}
AGM(1+\sqrt{3}, \sqrt{8}) = \frac{\sqrt[12]{2^{27}} \cdot \pi_2}{\sqrt[4]{3 \cdot 4 \cdot \pi_{6}^{3} \cdot \sqrt[2]{3} \cdot \pi_3}}
\end{equation}
Since 
\begin{equation*}
K(k) = \frac{\pi}{2 \cdot AGM(1, \sqrt{1-k^2})}
\end{equation*}
and 
\begin{equation*}
\lambda^{*}(x)^2 + \lambda^{*}(\frac{1}{x})^2 = 1
\end{equation*}
we have for $n \in \mathbb{N}$
\begin{equation}
AGM(1, \lambda^{*}(\frac{1}{n})) = \frac{\pi}{2 \cdot K(\lambda^{*}(n))}
\end{equation}
\\ \\
A superellipse is defined for $n \in \mathbb{N}$ and $a,b > 0$ as
\begin{equation}
C_n : \ \lvert \frac{x}{a} \lvert^{n} \ + \ \lvert \frac{y}{b} \lvert^{n} \ = 1
\end{equation}
It is well known, that the area is given by
\begin{equation}
A_{C_{n}} = 4 \cdot ab \cdot \frac{\Gamma(\frac{n+1}{n})^2}{\Gamma(\frac{n+2}{n})} = 4 \cdot ab \cdot \frac{1}{2n} \cdot \frac{\Gamma(\frac{1}{n})^2}{\Gamma(\frac{2}{n})}
\end{equation}
For the case $n = 2^{\nu}$ we calculate
\begin{equation*}
A_{C_{2^{\nu}}} = 4 \cdot ab \cdot \frac{1}{2^{\nu + 1}} \cdot \frac{2^{\nu - 1} \cdot \pi_{2^{\nu}} \cdot \sqrt{2^{\nu - 1} \cdot \pi_{2^{\nu - 1} \cdot \sqrt{...}}}}{\sqrt{2^{\nu - 2} \cdot \pi_{2^{\nu - 1} \cdot \sqrt{...}}}}
\end{equation*}
With ($\nu - 1$) nested square-roots, we obtain: \\
\begin{equation}
\forall \nu \in \mathbb{N}: \ A_{C_{2^{\nu}}} = \pi_{2^{\nu}} \cdot ab \cdot \sqrt{2 \sqrt{2 \sqrt{2 \sqrt{...}}}}
\end{equation}
where $0$ nested square-roots are to be understood as the factor $1$. \\ \\


\begin{thebibliography}{9}
\bibitem{gauss}
Carl Friedrich Gauss Werke Band III, p.230, Königliche Gesellschaft der Wissenschaften zu Göttingen (1866)

\bibitem{euler}
Carl Friedrich Gauss Werke Band X.1, p.150, B.G. Teubner, Leipzig (1917)

\bibitem{freitag}
Freitag, E., Busam, R. : Funktionentheorie 1, p.210, 4th edn. Springer, Heidelberg (2006)

\bibitem{koenigsberger}
Königsberger, K. : Analysis 1, p.360, 5th edn. Springer, Heidelberg (2001)

\bibitem{remmert}
Remmert, R., Schumacher, G. : Funktionentheorie 2, 3rd edn. Springer, Heidelberg (2007)

\bibitem{schneider}
Schneider, Th. : Zur Theorie der Abelschen Funktionen und Integrale. \textit{Journal für die reine und angewandte Mathematik}, Band 183, p.114 (1941)

\bibitem{siegel}
Siegel, C.L. : Topics In Complex Function Theory Vol. I, John Wiley Sons, (1969)

\bibitem{wolfram}
Weisstein, E.W., "Elliptic Integral Singular Value", Wolfram Mathworld 

\bibitem{wikipedia}
Wikipedia, "Particular values of the Gamma function", Last modified: 10th September 2021

\bibitem{nahin}
Nahin, P.J. : "Inside Interesting Integrals", Springer (2021)

\bibitem{lk}
Königsberger, L. : "Vorlesungen über die Theorie der hyperelliptischen Integrale", Teubner (1878)
\end{thebibliography}
\end{document}